\documentclass[11pt,a4paper]{amsart}
\usepackage{amscd,amssymb,amsopn,amsmath,amsthm,graphics,amsfonts,enumerate,verbatim,calc}
\usepackage[dvips]{graphicx}

\usepackage{color}
\newcommand{\syusei}[1]{{\color{red} #1}}
\renewcommand{\syusei}[1]{{#1}}

\usepackage{amssymb,amsmath}

\headsep=1cm \numberwithin{equation}{section}
\newtheorem{theorem}{Theorem}[section]

\newtheorem{proposition}[theorem]{Proposition}

\newtheorem{claim}[theorem]{Claim}

\theoremstyle{definition}
\newtheorem{definition}[theorem]{Definition}

\theoremstyle{remark}
\newtheorem{remark}[theorem]{Remark}
\newtheorem{example}[theorem]{Example}

\newcommand{\id}{\operatorname{id}}

\newcommand{\Frac}{\operatorname{Frac}}

\newcommand{\Min}{\operatorname{Min}}

\newcommand{\fm}{\frak{m}}

\newcommand{\fn}{\frak{n}}

\begin{document}
\title[An elementary proof of Cohen-Gabber theorem]
{An elementary proof of Cohen-Gabber theorem in the equal characteristic $p>0$ case}

\author[K. Kurano]{Kazuhiko Kurano}
\address{Department of Mathematics, School of Science and Technology,
Meiji University, Higashimata 1-1-1, Tama-ku, Kawasaki 214-8571, Japan}
\email{kurano@isc.meiji.ac.jp}

\author[K. Shimomoto]{Kazuma Shimomoto}
\address{Department of Mathematics, College of Humanities and Sciences, Nihon University, Setagaya-ku, Tokyo 156-8550, Japan}
\email{shimomotokazuma@gmail.com}


\thanks{The first author was partially supported by JSPS KAKENHI Grant 15K04828.
The second author was partially supported by Grant-in-Aid for Young Scientists (B) 25800028}


\begin{abstract}
The aim of this article is to give a new proof of Cohen-Gabber theorem in the equal characteristic $p>0$ case.
\end{abstract}

\maketitle

\section{Introduction}

Cohen proved the structure theorem on complete local rings in \cite{Cohen} and since then, it has been used as a basic tool in commutative algebra. Since our main concern is in rings of positive characteristic, let us recall its statement in the equal characteristic $p>0$ case, where $p$ is a prime number. Let $(A,\fm,k)$ be a complete local ring of dimension $d \ge 0$ and of equal characteristic $p>0$. In particular, $k$ is a field of characteristic $p$. Then there exists a coefficient field $\phi: k \to A$, together with a system of parameters $x_1,\ldots,x_d$ of $A$ such that there is a module-finite extension $\phi(k)[[x_1,\ldots,x_d]] \subset A$, where $\phi(k)[[x_1,\ldots,x_d]]$ is a complete regular local ring of dimension $d$. The aim of this article is to give a new and elementary proof of Cohen-Gabber theorem which is stated as follows:

\begin{theorem}[Cohen-Gabber]
\label{CohenGabber}
Assume that $(A,\fm,k)$ is a complete local ring of dimension $d \ge 0$ and of equal characteristic $p>0$ and let $\pi:A \to k=A/\fm$ be the quotient map. Then there exists a system of parameters $y_1,\ldots,y_d$ of $A$ and a ring map $\phi:k \to A$ such that the following hold: $\pi \circ \phi=\id_k$, the natural map
$$
\phi(k)[[y_1,\ldots,y_d]] \subset A
$$
is module-finite, and $\Frac(\phi(k)[[y_1,\ldots,y_d]]) \to \Frac(A/P)$ is a separable field extension for any minimal prime $P$ of $A$ such that $\dim A/P = d$.
\end{theorem}

The above theorem is seen as a strengthened version of Cohen structure theorem in that the module-finite extension $\phi(k)[[y_1,\ldots,y_d]] \subset A$ can be made to be generically \'etale when the local ring $R$ is reduced and equi-dimensional. Theorem \ref{CohenGabber} is formulated and proved by Gabber in \cite[Th\'eor\`eme 7.1]{GaOr} and \cite[Expos\'e IV, Th\'eor\`eme 2.1.1]{Book}. It plays a role in the proof of Gabber's alteration theorem with applications to \'etale cohomology. It is also essential in the proof of the Bertini-type theorem for reduced hyperplane quotients of complete local rings of characteristic $p>0$. This result is proved in \cite{OchShim}. Finally, we mention that there is a version of Theorem \ref{CohenGabber} for an affine domain over a perfect field (see \cite[Theorem 4.2.2]{SwHu}).

\section{Preliminaries}

We collect some facts that we shall use in this paper. All rings are assumed to be commutative and Noetherian with unity. A local ring is a Noetherian ring with a unique maximal ideal and it is denoted by the symbol $(A,\fm,k)$. We denote by $\Frac(A)$ the total ring of fractions of a commutative ring $A$.

\begin{definition}
\begin{enumerate}
\item
A \textit{coefficient field} of a complete local ring $(A,\fm,k)$ is a ring map $\phi:k \to A$ such that $\pi \circ \phi=\id_k$, where $\pi:A \to k=A/\fm$ is the quotient map. In particular, if a complete local ring has a coefficient field, this local ring contains the field of rationals $\mathbb{Q}$ or the finite field $\mathbb{F}_p$ for some 
prime number $p$.

\item
Let $K/k$ be a field extension such that the characteristic of $k$ is $p$ and let $\{\alpha_i\}_{i \in \Lambda}$ be a set of elements of $K$. Then say that $\{\alpha_i\}_{i \in \Lambda}$ is a \textit{p-basis} of $K$ over $k$, if $\{d\alpha_i\}_{i \in \Lambda}$ form a basis of the $K$-vector space $\Omega_{K/k}$, where $\Omega_{K/k}$ is the module of differentials of $K$ over $k$. Note that $K$ is a perfect field if $\Omega_{K/\mathbb{F}_p}=0$.
\end{enumerate}
\end{definition}

We recall a fundamental theorem on the existence of a coefficient field. For the proof, see \cite[Chapitre IX, \S~3, $\mathrm{n^o}$ 3, Th\'eor\`eme 1 b]{Bour}.

\begin{theorem}[Cohen]
\label{cohen}
Let $(A,\fm,k)$ be a complete local ring of equal characteristic $p>0$. Let $\{\alpha_i\}_{i \in \Lambda}$ be a set of elements of $A$ and let $\{\overline{\alpha_i}\}_{i \in \Lambda}$ be its image in $A/\fm=k$. If $\{\overline{\alpha_i}\}_{i \in \Lambda}$ is a $p$-basis of $k$ over $\mathbb{F}_p$, then there exists the unique coefficient field $\phi:k \to A$ such that $\phi(\overline{\alpha_i})=\alpha_i$ for each $i \in \Lambda$. Moreover, if $k$ is a perfect field, then $A$ has the unique coefficient field.
\end{theorem}

We recall Weierstrass Preparation Theorem. An element $f \in A[[X]]$ over a local ring $(A,\fm, k)$ is called a \textit{distinguished polynomial of degree n}, if we can write $f =X^n+a_{n-1}X^{n-1}+\cdots+a_1X+a_0$ for some integer $n \ge 0$ and $a_0,\ldots,a_{n-1} \in \fm$.

\begin{theorem}[Weierstrass Preparation Theorem] 
\label{Weierstrass}
Let $(A,\fm, k)$ be a complete local ring and let $B = A[[X]]$. Let $f=\sum_{i=0}^{\infty}a_i X^i \in B$ be a non-zero element with $a_i \in A$. If there exists a natural number $n \in \mathbb{N}$ such that $a_i \in \fm$ for all $i < n$ and $a_n \notin \fm$, then we have $f=u \cdot f_0$, where $u$ is a unit in $B$ and $f_0 \in B$ is a distinguished polynomial of degree $n$. Furthermore, $u$ and $f_0$ are uniquely determined by $f$.
\end{theorem}

See \cite{Ger} for a short proof of this theorem.

\begin{remark}
\label{preparation}
\begin{enumerate}
\item
Let $(A,\fm,k)$ be a $d$-dimensional complete local ring, let $\phi:k \to A$ be a coefficient field and let $x_1,\ldots,x_d \in \fm$. Then there is a natural injective ring map
$$
f:\phi(k)[[x_1,\ldots,x_d]] \subset A.
$$
Here, $\phi(k)[[x_1,\ldots,x_d]]$ is the image of the map 
 $\phi(k)[[X_1,\ldots,X_d]] \rightarrow A$ defined by $X_i \mapsto x_i$
for $i = 1, \ldots, d$, where $\phi(k)[[X_1,\ldots,X_d]]$ is the formal power series
ring over $\phi(k)$ with variables $X_1,\ldots,X_d$. The map $f$ is module-finite and $\phi(k)[[x_1,\ldots,x_d]]$ is isomorphic to a formal power series ring with $d$ variables
if and only if $x_1,\ldots,x_d$ is a system of parameters of $A$.

Suppopse that $x_1,\ldots,x_d, x_{d+1},\ldots,x_{d+h}$ generate $\fm$ such that $x_1,\ldots,x_d$ is a system of parameters of $A$. Then there is a module-finite ring map
$$
\phi(k)[[x_1,\ldots,x_d]] \subset A=\phi(k)[[x_1,\ldots,x_d, x_{d+1},\ldots,x_{d+h}]].
$$
Since $A/(x_1,\ldots,x_d)$
is a finite dimensional $\phi(k)$-vector space spanned by monomials on
$x_{d+1},\ldots,x_{d+h}$, $A$ is a finitely generated $\phi(k)[[x_1,\ldots,x_d]]$-module
also spanned by monomials on $x_{d+1},\ldots,x_{d+h}$ (cf. \cite[Theorem 8.4]{M}), i.e., 
$$
A = \phi(k)[[x_1,\ldots,x_d,x_{d+1},\ldots,x_{d+h}]]=\phi(k)[[x_1,\ldots,x_d]][x_{d+1},\ldots,x_{d+h}].
$$

\item
Let $R:=k[[X_1,\ldots,X_r]]$ be a formal power series ring over a field $k$ and choose $f \ne 0 \in R$. Write
$$
f=\sum_{i=0}^{\infty} b_i X_r^i
$$
with $b_i \in k[[X_1,\ldots,X_{r-1}]]$. Let $\fn$ be the maximal ideal of $k[[X_1,\ldots,X_{r-1}]]$ and assume that $b_0,\ldots,b_{\ell-1} \in \fn$ and $b_{\ell} \notin \fn$ for some $\ell>0$. By Theorem \ref{Weierstrass}, there is a unit $u \in R^{\times}$ together with a distinguished polynomial $g = X_r^{\ell}+a_{\ell-1}X_r^{\ell-1}+\cdots+a_{0}$ with $a_{0}, \ldots, a_{\ell-1} \in \fn$ such that
$$
f=u \cdot g.
$$
The ring injection $k[[X_1,\ldots,X_{r-1}]] \subset k[[X_1,\ldots,X_r]]$ induces an injection 
$$
S:=k[[X_1,\ldots,X_{r-1}]] \to R/(f)=:A .
$$
Here, $A$ is the $S$-free module with free basis $1$, $X_{r}$, \ldots, $X_{r}^{\ell-1}$. In particular, the ideal $(f)=(g)$ of $R$ does not contain any non-zero polynomial in $S[X_{r}]$ of degree strictly less than $\ell$.

We give an example to illustrate the situation. Let $A:=\mathbb{F}_p[[X,Y]]/(f)$ with $f=(X^p+tY^p)(X+1)$, where $t$ is transcendental over $\mathbb{F}_p$ and $X,Y$ are variables.

Then we have
$$ 
\frac{\partial f}{\partial X}=X^p+tY^p=0~\mbox{in}~A.
$$
We note that $f$ is a not a distinguished polynomial with respect to $X$.
\end{enumerate}
\end{remark}

\begin{remark}
\label{lemma}
Let $(A,\fm,k)$ be a $d$-dimensional local ring
such that $\fm$ is minimally generated by $d+h$ elements.
By the prime avoidance theorem, we can find $x_{1},\ldots,x_{d+h} \in \fm$ such that
\begin{itemize}
\item
$\fm=(x_{1},\ldots, x_{d+h})$, and
\item
any $d$ elements in $\{ x_{1},\ldots,x_{d+h} \}$ form a 
system of parameters of $A$.
\end{itemize}
\end{remark}

\section{Proof of Cohen-Gabber theorem}

We shall prove Theorem \ref{CohenGabber} in this section.

Let $(A,\fm,k)$ be a $d$-dimensional complete local ring containing a field of characteristic $p > 0$. Let $P_1, \ldots, P_r$ be the set of minimal prime ideals of $A$ of coheight $d$. If Theorem \ref{CohenGabber} is proved for $\tilde{A}:=A/P_1\cap \cdots \cap P_r$, then the same is true for $A$. \syusei{Indeed, $\tilde{A}$ is an equi-dimensional reduced local ring of dimension $d$. Then, if we can find a required coefficient field together with a system of parameters for $\tilde{A}$, we can lift them to $A$ by Theorem \ref{cohen}. Therefore, we may assume that}
\begin{equation}
\label{assumption}
\begin{tabular}{l}
\mbox{$(A,\fm,k)$ is a $d$-dimensional reduced equi-dimensional complete local ring}
\\
\mbox{containing a field of characteristic $p > 0$.}
\end{tabular}
\end{equation}

First we give a proof of the hardest case of Cohen-Gabber theorem.

\begin{proposition}
\label{hypersurface}
Let $(A,\fm,k)$ be a ring as in (\ref{assumption}).
Assume that the ideal $\fm$ is generated by $d+1$ elements. Then there exists a system of parameters $y_1,\ldots,y_d$ of $A$ and a ring map $\phi:k \to A$ such that the following hold: $\pi \circ \phi=\id_k$, the natural map
$$
\phi(k)[[y_1,\ldots,y_d]] \subset A
$$
is module-finite, and $\Frac(\phi(k)[[y_1,\ldots,y_d]]) \to \Frac(A/P)$ is a separable field extension for any minimal prime $P \subset A$.
\end{proposition}

\begin{proof}
We fix a coefficient field $\phi:k \to A$ together with a set of elements $x_1,\ldots,x_{d+1} \in \fm$ which satisfy the conclusion of Remark~\ref{lemma}. Then we have a module-finite injection
$$
k[[X_1,\ldots,X_d]] \to A
$$
by mapping $k$ to $\phi(k)$ and each $X_i$ to $x_i$. Since $A$ is reduced and equi-dimensional, after embedding $k[[X_1,\ldots,X_d]]$ to $R:=k[[X_1,\ldots,X_d,X_{d+1}]]$ in the natural way, we get a presentation:
$$
k[[X_1,\ldots,X_d]] \subset R \twoheadrightarrow R/(f) = A,
$$
where $f=f_1 \cdots f_r$ such that $f_i$ is irreducible for $i = 1, \ldots, r$.
Furthermore, we may assume that each $f_i$ is a distinguished polynomial with respect to $X_{d+1}$. Here, remark that, since $X_{1}, \ldots, X_{d}, f_{i}$ is a system of parameters of $R$ for $i = 1, \ldots, r$, each $f_{i}$ satisfies the assumption of Theorem~\ref{Weierstrass}. In summary,
\begin{enumerate}
\item
any $d$ elements in $\{x_1,\ldots,x_{d+1}\}$ form a system of parameters of $A$,

\item
$f=f_1 \cdots f_r$ is a factorization, where each $f_i$ is a prime element of $R$, and

\item
each $f_i$ is a distinguished polynomial with respect to $X_{d+1}$.
\end{enumerate}

We claim the following fact.

\begin{claim}
\label{sublemma1}
After replacing a coefficient field of $A$ and $x_1,\ldots,x_{d+1}$ if necessary, the following formula together with  (1), (2) and (3) above holds:
$$
\frac{\partial f_i}{\partial X_1} \ne 0 \ \mbox{in $R$}
$$
for all $i=1,\ldots,r$.
\end{claim}

Before proving this claim, let us see how Proposition \ref{hypersurface} follows from it. Since $f_i$ is a distinguished polynomial with respect to $X_{d+1}$, it follows from Claim \ref{sublemma1} that
\begin{equation}
\label{partial}
\frac{\partial f_i}{\partial X_1} \notin (f_i) \ \mbox{in $R$}
\end{equation}
since $\deg_{X_{d+1}}\frac{\partial f_i}{\partial X_1}$ is strictly less than 
$\deg_{X_{d+1}} f_{i}$ (see Remark \ref{preparation} (2)). Since $x_2,\ldots,x_{d+1}$ form a system of parameters of $A$ by (1), the composite ring map
$$
k[[X_2,\ldots,X_{d+1}]] \subset R \twoheadrightarrow R/(f) = A
\twoheadrightarrow R/(f_{i})
$$
is module-finite. Then by Remark~\ref{preparation}, we can find a unit $u_i \in R^{\times}$ such that $g_i:=f_i \cdot u_i$ is a distinguished polynomial with respect to $X_1$ for $i=1,\ldots,r$. Moreover, $g_i$ is a minimal polynomial of $x_1 \in A$ over $\Frac(k[[X_2,\ldots,X_{d+1}]])$. By Leibniz rule, we get
$$
\frac{\partial g_i}{\partial X_1}=\frac{\partial f_i}{\partial X_1}u_i+f_i \frac{\partial u_i}{\partial X_1}.
$$
Then by $(\ref{partial})$,
$$
\frac{\partial g_i}{\partial X_1} \notin (f_i)
$$
and in particular,
$$
\frac{\partial g_i}{\partial X_1} \ne 0 \ \mbox{in $R$}.
$$
Since $g_i$ is a distinguished polynomial with respect to $X_1$ satisfying the above, 
it follows that $g_i$ is a separable polynomial over $\Frac(k[[X_2,\ldots,X_{d+1}]])$. Therefore, the field extension
$$
\Frac(k[[X_2,\ldots,X_{d+1}]]) \subset \Frac(R/(f_i))=\Frac(R/(g_i))
$$
is finite separable and this proves Proposition~\ref{hypersurface}.

\begin{proof}[Proof of Claim \ref{sublemma1}]
The point is to make a good choice of a coefficient field of $A$. 
Consider the following condition for some $s \ge 1$:
\begin{equation}
\label{partial2}
\frac{\partial f_i}{\partial X_1} \ne 0~\mbox{for}~i=1,\ldots,s-1~\mbox{and}~\frac{\partial f_s}{\partial X_1}=0~\mbox{in}~R.
\end{equation}
Assume (\ref{partial2}). 
Then we shall prove that
\begin{equation}
\label{partial3}
\begin{tabular}{l}
\mbox{after replacing a coefficient field $\phi: k \to A$ and $X_1,\ldots,X_{d+1}$,}
\\
\mbox{$\displaystyle \frac{\partial f_i}{\partial X_1} \ne 0$ holds for every $i=1,\ldots,s$.}
\end{tabular}
\end{equation}

We prove (\ref{partial3}) in 2 steps below. Keep in mind that we assume (\ref{partial2}).

\begin{enumerate}
\item[$\bf{Step 1}$]
Let us assume that the following condition holds.
\begin{equation}
\label{casebycase1}
\mbox{There exists some $j \ge 2$ such that $\displaystyle \frac{\partial f_s}{\partial X_j} \ne0$
in $R$.}
\end{equation}
Write
$$
f_s=\sum_{a,b} F_{a,b} X_1^a X_j^b
$$
where
$$
\syusei{F_{a,b}:=F_{a,b}(X_2,\ldots,X_{j-1},X_{j+1},\ldots,X_{d+1})} \in k[[X_2,\ldots,X_{j-1},X_{j+1},\ldots,X_{d+1}]].
$$
By hypothesis (\ref{partial2}), if $F_{a,b} \ne 0$, then we have $p|a$. Define
$$
b_0:=\min\{b~|~p \nmid b~\mbox{and}~F_{a,b} \ne 0~\mbox{for some}~a \ge 0 \}
$$
and
$$
a_0:=\min\{a~|~F_{a,b_0} \ne 0\}.
$$
Note that $p|a_{0}$. We prove the following claim.
\begin{enumerate}\rm
\item[-]
Any choice of $d$ elements from the set $\{x_1,\ldots,x_{j-1},x_j-x_1^n,x_{j+1},\ldots,x_{d+1}\}$ forms a system of parameters of $A$ for $n \gg 0$.
\end{enumerate}
It suffices to take care of the $d$ elements $x_2,\ldots,x_{j-1},x_j-x_1^n,x_{j+1},\ldots,x_{d+1}$. 
Let $P \in \Min(A/(x_2,\ldots,x_{j-1},x_{j+1},\ldots,x_{d+1}))$. 
Then we have $x_j-x_1^n \notin P$ for $n \gg 0$. Indeed, if this is not the case, there exist $n_1$ and $n_2$ such that $n_1<n_2$ and $x_j-x_1^{n_1}, x_j-x_1^{n_2} \in P$. 
Then we would have that
$$
x_1^{n_1}(1-x_1^{n_2-n_1}) \in P~\mbox{and thus}~x_1 \in P.
$$
This is a contradiction. Hence the claim follows. By (\ref{partial2}), we may assume that the following condition
is satisfied.
\begin{equation}
\label{dist}
\mbox{For}~i=1,\ldots,s-1,~\mbox{the coefficient of}~X_1^{c_{i,1}} \cdots X_{d+1}^{c_{i,d+1}}~\mbox{in}~f_i~\mbox{is not zero and}~p \nmid c_{i,1}. 
\end{equation}
For the sequence $c_{1,1}, c_{2,1},\ldots,c_{s-1,1}$ as above, let us make a choice of an integer $q>0$ such that the following condition holds.

\begin{enumerate}\rm
\item[-]
Let $n:=qp+1$. Furthermore, $n$ is strictly greater than any element in the set $\{a_0,c_{1,1},c_{2,1},\ldots,c_{s-1,1}\}$, and any choice of $d$ elements from the set $\{x_1,\ldots,x_{j-1},x_j-x_1^n,x_{j+1},\ldots,x_{d+1}\}$ forms a system of parameters of $A$.
\end{enumerate}
We put
$$
Y_t:=X_t~\mbox{for}~t=1,\ldots,j-1,j+1,\ldots,d+1~\mbox{and}~Y_j:=X_j-X_1^n.
$$
Then
$$
g_i(Y_1,\ldots,Y_{d+1}):=f_i(Y_1,\ldots,Y_{j-1},Y_j+Y_1^n,Y_{j+1},\ldots,Y_{d+1}) \in R = k[[Y_1,\ldots,Y_{d+1}]]
$$
is still a distinguished polynomial with respect to $Y_{d+1}$. Let us look at the partial derivative of $g_i$ with respect to $Y_1$ for $i=1,\ldots,s-1$. Note that
$$
g_i(Y_1,\ldots,Y_{d+1})=Y_1^n \cdot h(Y_1,\ldots,Y_{d+1})+f_i(Y_1,\ldots,Y_{d+1})
$$
for some $h(Y_1,\ldots,Y_{d+1}) \in R = k[[Y_1,\ldots,Y_{d+1}]]$. Since the coefficient of $Y_1^{c_{i,1}} Y_2^{c_{i,2}} \cdots Y_{d+1}^{c_{i,d+1}}$ in $f_i(Y_1,\ldots,Y_{d+1})$ is not zero and $c_{i,1}<n$,
the coefficient of $Y_1^{c_{i,1}} Y_2^{c_{i,2}} \cdots Y_{d+1}^{c_{i,d+1}}$ in $g_i(Y_1,\ldots,Y_{d+1})$ turns out to be not zero.
Since $p \nmid c_{i,1}$ by $(\ref{dist})$, it follows that
\begin{equation}
\label{dist2}
\mbox{$\displaystyle \frac{\partial g_i}{\partial Y_1} \ne 0$ in $R$ for $i=1,\ldots,s-1$.}
\end{equation}
Next, define $G_{a,b}(Y_2,\ldots,Y_{j-1},Y_{j+1},\ldots,Y_{d+1}) \in k[[Y_2,\ldots,Y_{j-1},Y_{j+1},\ldots,Y_{d+1}]]$ by the following equation:
\begin{eqnarray*}
g_s(Y_1,\ldots,Y_{d+1}) & = & \sum_{a,b} F_{a,b}(Y_2,\ldots,Y_{j-1},Y_{j+1},\ldots,Y_{d+1})Y_1^a(Y_j+Y_1^n)^b \\
& = & \sum_{a,b}G_{a,b}(Y_2,\ldots,Y_{j-1},Y_{j+1},\ldots,Y_{d+1})Y_1^a Y_j^b.
\end{eqnarray*}
Then we claim that $G_{a_0+n,b_0-1} \ne 0$ by the choice of $n$. Indeed, if $p \mid b$,  $F_{a,b} Y_1^a(Y_j+Y_1^n)^b$ does not
contribute to $G_{a_0+n,b_0-1} Y_1^{a_0+n} Y_j^{b_0-1}$.
Since $n > a_0$, $F_{a,b} Y_1^a(Y_j+Y_1^n)^b$ does not
contribute to $G_{a_0+n,b_0-1} Y_1^{a_0+n} Y_j^{b_0-1}$ if $b > b_0$.
Therefore, only $F_{a_0,b_0} Y_1^{a_0}(Y_j+Y_1^n)^{b_0}$ contributes to
$G_{a_0+n,b_0-1} Y_1^{a_0+n} Y_j^{b_0-1}$.
We have $G_{a_0+n,b_0-1} =\binom{b_0}{1}F_{a_0,b_0}=b_0 F_{a_0,b_0} \neq 0$.

Since $p \nmid (a_0+n)$, it follows that
\begin{equation}
\label{dist3}
\frac{\partial g_s}{\partial Y_1} \ne 0 \ \mbox{in $R$}.
\end{equation}
Combining $(\ref{dist2})$ and $(\ref{dist3})$ together, we complete the proof of $(\ref{partial3})$ under the condition $(\ref{casebycase1})$. 

Now we move on to $\bf{Step 2}$.

\item[$\bf{Step 2}$]
Next, let us assume that the following condition holds.
\begin{equation}
\label{casebycase2}
\frac{\partial f_s}{\partial X_j}=0~\mbox{for all}~j=1,\ldots,d+1.
\end{equation}
In this case, the coefficient of some monomial on $X_1,\ldots,X_{d+1}$
in $f_{s}$ does not belong to $k^p$. If not, $f_s$ must be a $p$-th power of some element of $R = k[[X_1,\ldots,X_{d+1}]]$. In this case $A=R/(f)$ is not reduced, which contradicts to our hypothesis. Thus, we have $k^p \subsetneq k$ and in particular,
\begin{equation}
\label{dist4}
k~\mbox{is an infinite field}.
\end{equation}
Consider the set
\[
T=\{ (\ell_{1}, \ldots, \ell_{d+1}) \mid 
\mbox{the coefficient of $X_{1}^{\ell_1}X_2^{\ell_2} \cdots X_{d+1}^{\ell_{d+1}}$
in $f_{s}$ is not in $k^{p}$} \} .
\]
Let $(\ell'_{1}, \ldots, \ell'_{d+1})$ be an element of $T$ such that
$\ell'_{1}+ \cdots + \ell'_{d+1}$ is the minimum element 
in 
\[
\{ \ell_{1}+ \cdots + \ell_{d+1} \mid (\ell_{1}, \ldots, \ell_{d+1}) \in T \} .
\]
Note that
\begin{equation}
\label{lex}
\mbox{each of}~\ell'_1,\ldots,\ell'_{d+1}~\mbox{is divisible by}~p.
\end{equation}
Let $\alpha$ be the coefficient of $X_{1}^{\ell'_1}X_2^{\ell'_2} \cdots X_{d+1}^{\ell'_{d+1}}$
in $f_{s}$, and take a $p$-basis of $k/\mathbb{F}_p$:
\begin{equation}
\label{coefficient}
\{\alpha\} \cup \{\beta_{\lambda}\}_{\lambda \in \Lambda}.
\end{equation}

Let $\pi:k[[X_1]] \to k$ be the natural surjection. 
Let $\delta$ be an element in $k$.
Since $\{ \pi(\alpha+\delta X_1) \} \cup \{ \pi(\beta_{\lambda}) \}_{\lambda \in \Lambda}$ is a $p$-basis of $k/\mathbb{F}_p$, 
we have a map $\psi_{\delta}: k \to k[[X_1]]$ such that 
$\psi_{\delta}(\beta_{\lambda})=\beta_{\lambda}$ for $\lambda \in \Lambda$ and 
$\psi_{\delta}(\alpha) = \alpha+\delta X_1$ by Theorem~\ref{cohen}.
Here, $\psi_{\delta}$ naturally induces an isomorphism
$\tilde{\psi_{\delta}}:k[[X_1]] \to k[[X_1]]$ such that
$\tilde{\psi_{\delta}}(X_{1}) = X_{1}$ and $\tilde{\psi_{\delta}}|_{k} = 
\psi_{\delta}$, where $\tilde{\psi_{\delta}}|_{k}$ is the restriction of 
$\tilde{\psi_{\delta}}$ to $k$.
Let $\phi_{\delta}$ be the composite map of
\[
k \subset k[[X_{1}]] \stackrel{ {\tilde{\psi_{\delta}}}^{-1} }{\longrightarrow}
k[[X_{1}]] .
\]
Let
$$
\Phi_{\delta}: R \to R
$$
be a ring isomorphism such that $\Phi_{\delta}(X_{i}) = X_{i}$ for $i = 1, \ldots, d+1$
and $\Phi_{\delta}|_{k[[X_{1}]]} = \tilde{\psi_{\delta}}$.
We have the following commutative diagram
\[
\begin{array}{ccc}
k & & \\
\downarrow & \hphantom{\scriptstyle \phi_{\delta}} \searrow {\scriptstyle \phi_{\delta}} & \\
k[[X_{1}]] & \stackrel{\tilde{\psi_{\delta}}}{\longleftarrow} & k[[X_{1}]] \\
\downarrow & & \downarrow \\
R & \stackrel{\Phi_{\delta}}{\longleftarrow} & R
\end{array}
\]
where the vertical maps are the natural inclusions.

Considering the composite map of
\[
k \stackrel{\phi_{\delta}}{\longrightarrow} k[[X_{1}]] \subset R = k[[X_{1}, \ldots, X_{d+1}]]
\]
as a new coefficient field of $R$, we have an isomorphism
\[
R/(f_{i}) \simeq R/(\Phi_{\delta}(f_{i})) .
\]

We claim the following.
\begin{enumerate}\rm
\item[-]
For $i=1,\ldots,s-1$, one can present the coefficient of $X_1^{c_{i,1}} \cdots X_{d+1}^{c_{i,d+1}}$ in $\Phi_{\delta}(f_i)$ in the form $\xi_i(\delta)$, where $\xi_i(X) \in k[X]$ and $\xi_i(0) \ne 0$, where the sequence $c_{i,1},\ldots,c_{i,d+1}$ is given as in (\ref{dist}).
\end{enumerate}
Let us prove this claim. Let $\xi_{i}(\delta)$ be the coefficient of 
$X_1^{c_{i,1}} \cdots X_{d+1}^{c_{i,d+1}}$ in $\Phi_{\delta}(f_i)$.
Note by (\ref{dist}) that $\xi_{i}(0) \neq 0$ since $\Phi_{0}$ is the identity.
We shall prove that $\xi_{i}(\delta)$ is a polynomial function on $\delta$. Pick an element $c \in k \subseteq R$, 
where $k$ embeds into $R = k[[X_1,\ldots,X_{d+1}]]$ in the natural way. 
Then we have $\Phi_{\delta}(c)=\tilde{\psi_{\delta}}(c) \in k[[X_1]]$ and 
$\tilde{\psi_{\delta}}(c)-c$ is divisible by $X_1$. 
So we can write
\begin{equation}
\label{eta}
\Phi_{\delta}(c)=c+\sum_{i=1}^{\infty}\eta_{c,i}(\delta)X_1^i ,
\end{equation}
where $\eta_{c,i}(\delta) \in k$ for $i \ge 1$. It is sufficient to prove that each $\eta_{c,i}(\delta)$ is a polynomial with respect to $\delta$.

For a fixed integer $e >0$, note that the set
$$
\left\{ \alpha^q \cdot \prod_{\lambda \in \Lambda} \beta_{\lambda}^{q_\lambda}
\ \left| \ 
\begin{tabular}{l}
\mbox{$q,q_{\lambda}=0,1,\ldots,p^e-1$,} \\
\mbox{$q_{\lambda} = 0$ except for finitely many $\lambda \in \Lambda$}
\end{tabular}
 \right\} \right.
$$
forms a basis of the $k^{p^e}$-vector space $k$. We use the symbol $\underline{q}_{\lambda}$ to denote a vector $\{ q_{\lambda} \}_{\lambda \in \Lambda}$, where
$q_{\lambda} = 0$ except for finitely many $\lambda \in \Lambda$.
Suppose
$$
\syusei{c=\sum_{q,\underline{q}_{\lambda}=0}^{p^e-1} (d_{q,\underline{q}_{\lambda}})^{p^e} \cdot \alpha^q \cdot (\prod_{\lambda \in \Lambda} \beta_{\lambda}^{q_\lambda})}
$$
for $d_{q,\underline{q}_{\lambda}} \in k$. Applying the map $\Phi_{\delta}$, we have
$$
\Phi_{\delta}(c)=\sum_{q,\underline{q}_{\lambda}=0}^{p^e-1} \Phi_{\delta}(d_{q,\underline{q}_{\lambda}})^{p^e} \cdot (\alpha+\delta X_1)^q \cdot (\prod_{\lambda \in \Lambda} \beta_{\lambda}^{q_\lambda}),
$$
where
$$
\Phi_{\delta}(d_{q,\underline{q}_{\lambda}})^{p^e}=\big(d_{q,\underline{q}_{\lambda}}+X_1 \cdot \gamma \big)^{p^e}=(d_{q,\underline{q}_{\lambda}})^{p^e}+X^{p^e}_1 \cdot \gamma^{p^e}
$$
for some $\gamma \in R$. Letting $\eta_{c,i}(\delta)$ be as in $(\ref{eta})$, it follows that, if \syusei{$e \ge 0$ is an integer satisfying $p^e>i$}, then
$\eta_{c,i}(\delta)$ is the coefficient of $X_1^i$ in
\[
\sum_{q,\underline{q}_{\lambda}=0}^{p^e-1} (d_{q,\underline{q}_{\lambda}})^{p^e} \cdot  (\alpha+\delta X_1)^q \cdot (\prod_{\lambda \in \Lambda} \beta_{\lambda}^{q_\lambda}) .
\]
This description shows that $\eta_{c,i}(\delta)$ is a polynomial with respect to $\delta$. Therefore, $\xi_{i}(\delta)$ is also a polynomial function on $\delta$ for $i = 1, \ldots, s-1$.

Let us write
$$
\xi(x):=\xi_1(x) \cdots \xi_{s-1}(x) \in k[x].
$$
Then since $\xi(0) \ne 0$, there exists $\delta \in k^{\times}$ such that $\xi(\delta) \ne 0$ due to the fact that $k$ is an infinite field $(\ref{dist4})$. Now we are going to finish the proof of Claim~\ref{sublemma1}.

\begin{enumerate}\rm
\item[-]
For $i=1,2,\ldots,s-1$, we have
$$
\frac{\partial \Phi_{\delta}(f_i)}{\partial X_1} \ne 0,
$$
since the coefficient of $X^{c_{i,1}}_1 X^{c_{i,2}}_2 \cdots X^{c_{i,d+1}}_{d+1}$ in $\Phi_{\delta}(f_i)$ is $\xi_{i}(\delta)$, which is not $0$ by the choice of $\delta$.
\item[-]
For $i=s$, we shall prove that 
the coefficient of $X^{\ell'_{1}+1}_1 X^{\ell'_{2}}_2 \cdots X^{\ell'_{d+1}}_{d+1}$ in $\Phi_{\delta}(f_s)$  is not zero.
Put
\[
f_{s} = \sum_{\ell_{1}, \ldots, \ell_{d+1}} c_{\ell_{1}, \ldots, \ell_{d+1}}
X_{1}^{\ell_{1}} X_{2}^{\ell_{2}} \cdots X_{d+1}^{\ell_{d+1}}, 
\]
where $c_{\ell_{1}, \ldots, \ell_{d+1}} \in k$.
Then, we have
\begin{eqnarray*}
\Phi_{\delta}(f_{s}) & = & 
\sum_{\ell_{1}, \ldots, \ell_{d+1}} \Phi_{\delta}(c_{\ell_{1}, \ldots, \ell_{d+1}})
X_{1}^{\ell_{1}} X_{2}^{\ell_{2}} \cdots X_{d+1}^{\ell_{d+1}} \\
& = & 
\sum_{\ell_{1}, \ldots, \ell_{d+1}} \left( c_{\ell_{1}, \ldots, \ell_{d+1}} 
+ \sum_{i = 1}^{\infty} \eta_{c_{\ell_{1}, \ldots, \ell_{d+1}},i}(\delta)X_{1}^{i} \right)
X_{1}^{\ell_{1}} X_{2}^{\ell_{2}} \cdots X_{d+1}^{\ell_{d+1}} \\
& = & f_{s} + \sum_{\ell_{1}, \ldots, \ell_{d+1}} \sum_{i = 1}^{\infty}
\eta_{c_{\ell_{1}, \ldots, \ell_{d+1}},i}(\delta)
X_{1}^{\ell_{1}+i} X_{2}^{\ell_{2}} \cdots X_{d+1}^{\ell_{d+1}} .
\end{eqnarray*}
Therefore, the coefficient of $X_{1}^{\ell'_{1}+1} X_{2}^{\ell'_{2}} \cdots X_{d+1}^{\ell'_{d+1}}$ in $\Phi_{\delta}(f_{s})$ is
\[
\sum_{i = 1}^{\ell'_{1}+1} \eta_{c_{\ell'_{1}+1-i, \ell'_{2}, \ldots, \ell'_{d+1}},i}
(\delta), 
\]
since the coefficient of $X_{1}^{\ell'_{1}+1} X_{2}^{\ell'_{2}} \cdots X_{d+1}^{\ell'_{d+1}}$ in $f_{s}$ is zero by (\ref{casebycase2}) and (\ref{lex}).
If $\eta_{c_{\ell'_{1}+1-i, \ell'_{2}, \ldots, \ell'_{d+1}},i}
(\delta) \neq 0$, then this implies that \syusei{$c_{\ell'_{1}+1-i, \ell'_{2}, \ldots, \ell'_{d+1}} \neq 0$ and thus, $\ell'_{1}+1-i$ is divisible by $p$ by (\ref{casebycase2})}.
So, we assume that $i \equiv 1 \mod p$ by (\ref{lex}).
If $i = qp+1$ with $q>0$, then $c_{\ell'_{1}+1-i, \ell'_{2}, \ldots, \ell'_{d+1}}
\in k^{p}$ by the definition of $(\ell'_{1}, \ell'_{2}, \ldots, \ell'_{d+1})$.
\syusei{Under the notation as in $(\ref{eta})$, note that $\eta_{\gamma, i}(\delta) = 0$ if $p \nmid i$ and $\gamma \in k^{p}$, because $\Phi_{\delta}(\gamma)$ has a $p$-th root in $R$.} Therefore, we have 
$\eta_{c_{\ell'_{1}+1-i, \ell'_{2}, \ldots, \ell'_{d+1}},i}(\delta) = 0$ for any integer $i \ge 2$.
Then, the coefficient of $X_{1}^{\ell'_{1}+1} X_{2}^{\ell'_{2}} 
\cdots X_{d+1}^{\ell'_{d+1}}$ in $\Phi_{\delta}(f_{s})$ is
\[
\eta_{c_{\ell'_{1}, \ell'_{2}, \ldots, \ell'_{d+1}},1}(\delta) =
\eta_{\alpha,1}(\delta) = \delta \neq 0,
\]
\syusei{where $\alpha \in k$ is as in $(\ref{coefficient})$.} Hence, we obtain
$$
\frac{\partial \Phi_{\delta}(f_s)}{\partial X_1} \ne 0 .
$$
\end{enumerate}
We complete the proof of (\ref{partial3}) under the condition (\ref{casebycase2}).
\end{enumerate}
We have completed the proof of Claim~\ref{sublemma1}.
\end{proof}
We have completed the proof of Proposition~\ref{hypersurface}, which is the hypersurface case of Cohen-Gabber theorem.
\end{proof}

As noted in (\ref{assumption}), it suffices to prove Cohen-Gabber theorem (Theorem \ref{CohenGabber}) in the reduced equi-dimensional case.

\begin{proof}[Proof of Theorem \ref{CohenGabber}]
Let $\fm=(x_1,\ldots,x_d,x_{d+1},\ldots,x_{d+h})$ such that $x_1,\ldots,x_d$ is a system of parameters of $A$ (see Remark~\ref{lemma}).
We shall prove the reduced equi-dimensional case of Theorem~\ref{CohenGabber} by induction on $h$.

$\mathbf{h=0}$:
Since $\fm$ is generated by $d$ elements, we have $A=k[[x_1,\ldots,x_d]]$ and we are done in this case.

$\mathbf{h=1}$:
This is already established as in Proposition~\ref{hypersurface}.

$\mathbf{h \ge 2}$:
With notation as above, fix a coefficient field $\phi:k \to A$.
Then $A=\phi(k)[[x_1,\ldots,x_{d+h}]]=\phi[[x_1,\ldots,x_d]][x_{d+1},\ldots,x_{d+h}]$ by Remark~\ref{preparation} (1). We consider the following commutative diagram of complete local rings:
$$
\footnotesize{
\begin{CD}
A=\phi(k)[[x_1,\ldots,x_{d+h}]] @<<< D:=\phi'(k)[[y_1,\ldots,y_d,x_{d+2},\ldots,x_{d+h}]] @<<< E:=\phi''(k)[[z_1,\ldots,z_d]] \\
@AAA @AAA @. \\
B:=\phi(k)[[x_1,\ldots,x_{d+1}]] @<<< C:=\phi'(k)[[y_1,\ldots,y_d]] \\
@AAA \\
\phi(k)[[x_1,\ldots,x_d]] \\
\end{CD}}
$$
We explain the structure of the above diagram.

\begin{enumerate}\rm
\item[-]
Let $B$ be the subring $\phi(k)[[x_1,\ldots,x_{d+1}]]$ of $A$.
After applying the case $\mathbf{h=1}$ to $B$, we can find a coefficient field $\phi':k \to B$ together with a system of parameters $y_1,\ldots,y_d$ to get a formal power series ring $C=\phi'(k)[[y_1,\ldots,y_d]]$ such that $\Frac(C) \to \Frac(B/P)$ is a separable field
extension
for any minimal prime ideal $P$ of $B$. 
Let $D$ be the subring $\phi'(k)[[y_1,\ldots,y_d,x_{d+2},\ldots,x_{d+h}]]$ of $A$.
Note that $C \subset D \subset A$.

\item[-]
Since the maximal ideal of $D$ is generated by at most $d+h-1$ elements, we can find, by induction hypothesis on $h$, a coefficient field $\phi'':k \to D$ together with a system of parameters $z_1,\ldots,z_d$ to get a formal power series ring $E=\phi''(k)[[z_1,\ldots,z_d]]$ such that $\Frac(E) \to \Frac(D/Q)$ is a separable field extension
for any minimal prime ideal $Q$ of $D$.
\end{enumerate}

All the maps appearing in the diagram are injective and module-finite. We claim that $E \to A$ satisfies the conclusion of Cohen-Gabber theorem. To see this, fix a minimal prime $P \subset A$ and form the following commutative diagram of quotient fields.
$$
\begin{CD}
\Frac(A/P) \\
@| \\
\Frac(D/D \cap P)(x_1,\ldots,x_{d+1}) @<f_1<< \Frac(D/D \cap P) @<f_2<< \Frac(E) \\
@AAA @AAA \\
\Frac(C)(x_1,\ldots,x_{d+1}) @<f_3<< \Frac(C) \\
@| \\
\Frac(B/B \cap P)
\end{CD}
$$
Note that $E$ and $C$ are domains over which $A$ is module-finite and torsion free, so we have $E \cap P=(0)$ and $C \cap P=(0)$. We need to prove that $\Frac(E) \to \Frac(A/P)$ is a separable field extension. By construction, $f_3$ is separable and $f_1$ is obtained by adjoining $x_{1},\ldots,x_{d+1}$ to $\Frac(D/D \cap P)$. Hence $f_1$ is separable. That is, we proved that $f_1 \circ f_2$ is separable.
\end{proof}

We end this paper with the following example. For a ring map $R \to S$, let $\Omega_{S/R}$ denote the module of differentials of $S$ over $R$. It is regarded as an $S$-module.

\begin{example}
Consider the following two integral domains;
\begin{eqnarray*}
A & = & \mathbb{F}_{p}(t)[[X,Y]]/(tX^{p} + Y^{p}) , \\
B & = & \mathbb{F}_{p}(t)[X,Y]/(tX^{p} + Y^{p}) ,
\end{eqnarray*}
where $t$ is transcendental over $\mathbb{F}_{p}$ and
$X$, $Y$ are variables.

We have $\Omega_{B/\mathbb{F}_{p}(t)} = BdX + BdY \simeq B^{\oplus 2}$.
Here, assume that $\mathbb{F}_{p}(t)[z] \hookrightarrow B$ is a module-finite map
for some $z \in B$.
Note that
\[
\Omega_{B/\mathbb{F}_{p}(t)[z]} = \Omega_{B/\mathbb{F}_{p}(t)}/Bdz .
\]
Since it is not a torsion $B$-module, $\Frac(B)$ is not separable 
over $\Frac(\mathbb{F}_{p}(t)[z])$.

Let $w$ be any non-zero element in the maximal ideal of $A$.
Then, $\mathbb{F}_{p}(t)[[w]] \rightarrow A$ is a module-finite extension.
Then, we have
\[
\Omega_{A/\mathbb{F}_{p}(t)[[w]]} = AdX + AdY/Adw ,
\]
and it is not a torsion $A$-module.
Hence, $\Frac(A)$ is not separable 
over $\Frac(\mathbb{F}_{p}(t)[[w]])$.

On the other hand, put $s = t+X \in A$.
Then, $\mathbb{F}_{p}(s)$ is another coefficient field of $A$, and
\[
A = \mathbb{F}_{p}(s)[[X,Y]]/((s-X)X^{p} + Y^{p}) .
\]
Then, $\mathbb{F}_{p}(s)[[Y]] \rightarrow A$ is module-finite and
$\Frac(A)$ is a separable field extension over $\Frac(\mathbb{F}_{p}(s)[[Y]])$.
\end{example}

\end{document}